\documentclass[nonatbib,1p]{elsarticle}
\title{An Empirically Fast Las Vegas Algorithm for Algebraic Shifting}

\author[1]{Antony Della Vecchia}
\author[1,2]{Michael Joswig}
\author[1]{Fabian Lenzen}
\affiliation[1]{organization={Technische Universität Berlin},country={Germany}}
\affiliation[2]{organization={MPI MiS}, city={Leipzig}, country={Germany}}
\usepackage{xstring}

\makeatletter
\PassOptionsToPackage{sortcites,isbn=false}{biblatex}
\PassOptionsToPackage{inline}{enumitem}
\PassOptionsToPackage{hyperfootnotes=false,colorlinks}{hyperref}
\PassOptionsToPackage{capitalise}{cleveref}
\usepackage{bbding}
\usepackage{rotating}
\usepackage{amsthm,amsfonts,mathtools}

\usepackage[ruled,nosemicolon,vlined]{algorithm2e}
\SetKw{Break}{break}
\SetKw{Return}{return}
\SetKw{Continue}{continue}
\SetKwRepeat{DoWhile}{do}{while}
\SetKwRepeat{DoUntil}{do}{until}
\newcommand{\True}{\textup{\texttt{true}}}
\newcommand{\False}{\textup{\texttt{false}}}

\usepackage[legacy]{csvsimple}
\usepackage{siunitx}
\input{definitions}

\renewcommand\OSCAR{\textup{\texttt{OSCAR}}\xspace}

\renewcommand\Macaulay{\textup{\texttt{Macaulay2}}\xspace}

\DeclareSourcemap{
	\maps{
		\map{\step[fieldset=urldate, null]}
		\map{\step[fieldset=issn, null]}
		\map{
			\step[fieldset=eventtitle, null]
			\step[fieldset=eventdate, null]
			\step[fieldset=venue, null]
		}
		\map{
			\step[fieldsource=date, match={\regexp{([0-9]{4})[0-9-]*}}, replace={\regexp{$1}}]
		}
		\map{
			\step[fieldsource=doi,final]
			\step[fieldset=isbn,null]
			\step[fieldset=issn,null]
			\step[fieldset=url,null]
		}
		\map{
			\step[fieldsource=isbn,final]
			\step[fieldsource=issn,final]
			\step[fieldset=url,null]
		}
	}
}

\newcommand\Rot[1]{\rotatebox{90}{#1}}%
\newcommand\CS[1]{\csname #1\endcsname}%
\newcommand\Tick[1]{\resizebox{!}{\heightof{t}}{\ifthenelse{#1=1}{\Checkmark}{\XSolidBrush}}}%
\newcommand\EarlyReturn[2][Lv]{\ifthenelse{\equal{\CS{#2#1maxLenBefore}}{n/a}}{*}{}}%
\newcommand\TimeOrDash[2][]{%
	\IfSubStr{\CS{#2Trials}}{>}%
	{{--}}%
	{\tablenum[table-auto-round,table-format=#1]{\CS{#2Time}}}%
}
\newcommand\FailedShift[2][Lv]{\IfSubStr{\CS{#2#1Trials}}{>}{${}^\dagger$}{}}%
\DeclareSIUnit{\noop}{\kern 0pt}%
\newcommand{\Mem}[1]{%
	\IfInteger{#1}{%
		\makebox[\widthof{\SI{10}{\giga\byte}}][r]{\SI[exponent-mode=engineering, exponent-to-prefix = true, round-mode=places, round-precision=0]{#1}{\noop}}
	}{#1}%
}%
\ExplSyntaxOn
\NewDocumentCommand{\FormatGroup}{m}{
	\tl_set:Nx \l_tmpa_tl {#1}
	\regex_replace_all:nnN {Z/1} {0} \l_tmpa_tl
	\regex_replace_all:nnN {Z/([0-9])} {Z/\1Z} \l_tmpa_tl
	\regex_replace_all:nnN {Z} {\c{mathbb}{Z}} \l_tmpa_tl
	\tl_use:N \l_tmpa_tl
}
\ExplSyntaxOff

\usepackage{varwidth}
\begin{document}
	
\begin{abstract}
	Improved algorithms for computing (partial and full) exterior algebraic shifts of hypergraphs and simplicial complexes are presented.
	The main benefit is in positive characteristic.
	Experiments with an implementation in \OSCAR\ with various inputs such as bipartite graphs and triangulations of two and three dimensional manifolds
	show that the method considerably extends for which simplicial complexes exterior algebraic shifts can be computed in practice.
\end{abstract}

\maketitle


\section{Introduction} 
Algebraic shifting comprises a variety of powerful techniques to replace a finite simplicial complex $K$ by a simpler complex which still retains key properties of $K$.
In particular, the f-vector, which counts the number of faces per dimension, and the Betti numbers (with respect to a fixed field $\EE$ of arbitrary characteristic) remain the same.
Algebraic shifting was introduced by Kalai \cite{Kalai84,Kalai02} and further developed by Björner--Kalai \cite{BjornerKalai:1988}, Aramova--Herzog \cite{Aramova+Herzog:2000}, Nevo \cite{Nevo:2005}, Murai--Hibi \cite{Murai+Hibi:2009} and others.
Here we are concerned with \emph{exterior shifting}, i.e., algebraic shifting in an exterior algebra over $\EE$.
In this setting, computing a shift seems to be straightforward.
It just requires finding the row echelon form of some matrix with coefficients in $\EE$.

The computational challenge is the field $\EE$ itself, which must be large enough to accommodate a certain \enquote{generic linear transformation}.
In practice, $\EE$ is a sufficiently large purely transcendental extension of a prime field $\FF$; i.e., $\EE=\FF(x_1,\ldots,x_k)$ is a field of rational functions.
So the bulk of the computational cost comes from the arithmetic in $\EE$.
In fact, no algorithm is known to compute the exterior shift of an arbitrary simplicial complex $K$ in polynomial time.
However, in characteristic zero there is a known way out.
A transformation matrix with coefficients picked at random in $\EE=\RR$ is sufficiently generic with high probability.
This yields a fast Monte Carlo algorithm which is standard \cite[\S2.6]{Kalai02};
an implementation based on \Macaulay \cite{M2} by Keehn can be found here \cite{M2ExtShifting}.
Kalai raised the question if there is a general deterministic or at least a Las Vegas polynomial time algorithm for exterior shifting \cite[\S2.6]{Kalai02}.
Recently, Keehn and Nevo \cite{KeehnNevo:2024} found deterministic polynomial algorithms in characteristic zero for triangulations of the torus, the real projective plane and the Klein bottle.

The purpose of this article is to describe algorithms for exterior shifting which are faster in practice than previous techniques.
Our best method is of \emph{Las Vegas type}, i.e., it is a randomized algorithm, which always gives the correct answer, if it terminates.
However, with low probability, it may run forever; in characteristic zero that probability is actually zero.
In contrast, a \emph{Monte Carlo algorithm} always terminates, but it may give a wrong answer with low probability.
All our algorithms are implemented in the new open source computer algebra system \OSCAR \cite{Oscar,OSCAR-book}.
While there is an advantage in all cases, the real benefit is in positive characteristic, where the standard Monte Carlo algorithm does not work.
The core idea rests on two independent contributions.
First, we find a transcendental extension of the prime field which is substantially smaller than the naive choice.
So, to fit a generic linear transformation, $\EE$ can be chosen as $\EE=\FF(x_1,\ldots,x_k)$ where $k=n^2$ and $n$ is the number of vertices of the complex $K$.
Exploiting the Bruhat decomposition of the general linear group $\GL(n,\EE)$ we can reduce $k$ to $n(n-1)/2$, which is a significant improvement.
This idea was developed in our recent work on partial algebraic shifting, which is algebraic shifting with not necessarily generic transformations \cite{VecchiaJoswigLenzen:2024}.
In this way we obtain interesting shifts in positive characteristic \cite[Example 44]{VecchiaJoswigLenzen:2024}, which were previously out of reach.
The relevant findings from that article are summarized in \cref{sec:shifting}.
Secondly, the actual linear algebra over $\EE$ can be organized in a way such that many redundant arithmetic operations are avoided.
In practice we obtain a reasonably fast procedure for deciding if a shift is actually correct (\cref{algo:verify} and \cref{thm:algo:a}); this is useful in every characteristic.
The latter strategy is explained in \cref{sec:shift-algo}.
This decision procedure directly yields a Las Vegas algorithm.
\Cref{sec:implementation} covers the implementation details.
One additional details is a tailored lazy evaluation scheme for the row echelon form (\cref{alg:lazy reduction}).
This approach allows us to avoid further arithmetic operations.
We report on computational experiments with \OSCAR\cite{Oscar,OSCAR-book} in \cref{sec:experiments}.

This article is the full version of an extended abstract that was published in the proceedings of ISSAC 2025\footfullcite{DellaVecchiaJoswigEtAl:2025}.
The work on this version of the text lead to a further substantial improvement of the algorithms and their implementations.
The difference is well visible in our experiments.
Notably, we can now compute shifts of certain simplicial complexes in dimensions 3 and 4, which were totally out of reach before; cf.\ \cref{sec:balls-and-spheres}.
Moreover, recently, Bulavka, Nevo and Peled studied algebraic shifting of random triangulations of surfaces \cite{BulavkaNevoPeled25}.
For various scenarios they showed that the homology of the shifts of these random complexes is particularly nice.
For these \enquote{homology lex segments} our algorithms also work very well, provided that the homology is concentrated in the top dimension.

\paragraph*{Acknowledgments}
  We are indebted to all developers of \OSCAR; without their numerous contributions our entire project would have been impossible.
  In particular, we thank Claus Fieker for enlightening discussions on linear algebra over fields of functions.
  We would also like to thank three anonymous ISSAC 2025 reviewers for their comments.

  The work of all authors is funded by the Deutsche Forschungsgemeinschaft (DFG, German Research Foundation).
  Specifically, ADV and MJ were supported by \enquote{MaRDI (Mathematical Research Data Initiative)} (DFG NFDI 29/1, project ID 460135501);
  MJ was supported by \enquote{Symbolic Tools in Mathematics and their Application} (DFG TRR 195, project ID 286237555);
  MJ and FL were supported by The Berlin Mathematics Research Center MATH$^+$ (DFG EXC-2046/1, project ID 390685689).

\section{Full and Partial Algebraic Shifting} 
\label{sec:shifting}
For $n \in \NN$, denote by $\binom{[n]}{k}$ the set of $k$-element subsets of $[n] \coloneqq \{1,2,\dots,n\}$.
A \emph{$k$-uniform hypergraphs} on $n$ elements is a nonempty subset of $\binom{[n]}{k}$.
The elements of $\binom{[n]}{k}$ are called \emph{hyperedges}.
All our hypergraphs are uniform.

\begin{definition}
	\label{def:domination-order}
	Let $(S, \leq)$ be a totally ordered set.
	The total order $\leq$ on $S$ induces a partial order on $\binom{S}{k}$, called \emph{domination order} with respect to $\leq$, 
	given by $\{a_1 < \dotsb < a_k\} \leq \{b_1 < \dotsb < b_k \}$ if $a_i \leq b_i$ for all $i$.
\end{definition}

A hypergraph $S \subseteq \binom{[n]}{k}$ is \emph{shifted} if it is an initial set with respect to the domination order on $\binom{[n]}{k}$;
i.e., if $\sigma \in S$ and $\rho \leq \sigma$, then also $\rho \in S$.
Verifying if a given uniform hypergraph $S \subseteq \binom{[n]}{k}$ is shifted is immediate.
For $\sigma \in S$ and $i$, $j \in [n]$, let
\begin{equation}
	\label{eq:replace-vertex}
	\RepVert{\sigma}{i}{j} \coloneqq \begin{cases}
		\sigma \setminus \{i\} \cup \{j\} & \text{if $i \in \sigma j \notin \sigma$}, \\
		\sigma                            & \text{otherwise}.
	\end{cases}
\end{equation}
Now \cref{alg:deciding shiftedness} verifies of a given hypergraph $S$ is shifted in time $\mathcal{O}(k\abs{S})$.

\begin{algorithm}[bp]
	\caption{Deciding shiftedness of a uniform hypergraph.}
	\label{alg:deciding shiftedness}
	\KwIn{$S \subseteq \binom{[n]}{k}$}
	\KwOut{\True\enspace if $S$ is shifted, \False \enspace otherwise}
	\For{$\sigma \in S$}{
		\For{$i \in \sigma$, $i > 1$}{
			\lIf{$\RepVert{\sigma}{i}{i-1} \notin \sigma$}{%
				\Return\False
			}
		}
	}
	\Return\True
\end{algorithm}

The following construction assigns to every $S \subseteq \binom{[n]}{k}$ a shifted hypergraph $\Delta(S)$ that shares several combinatorial properties with $S$.
Let $\EE \supseteq \FF$ be a field extension and $g\in\GL(n,\EE)$ be an invertible matrix.

\begin{definition}
	\label{def:lex-order}
	Let $(S, \leq)$ be a totally ordered set.
	The \emph{lexicographic order} on $\binom{S}{k}$ is the total order
	given by $s \lex\leq t$ if $\min (s \symdiff t) \in s$ with the symmetric difference $\symdiff$.
	It refines the domination order from \cref{def:domination-order}.
\end{definition}

In particular, the natural total order on $[n]$ induces a lexicographic order $\lex\leq$ on $\binom{[n]}{k}$.
It allows us to identify $\binom{[n]}{k}$ with $\bigl[\binom{n}{k}\bigr]$.
Then the $k$th \emph{compound matrix} $g^{\wedge k}$ of $g$ is the $\binom{n}{k} \times \binom{n}{k}$-matrix
with $g^{\wedge k}_{\sigma\tau} = \det (g_{ij})_{i \in \sigma, j \in \tau}$ for $\sigma, \tau \in \binom{[n]}{k}$.
We denote the row and column of $g^{\wedge k}$ corresponding to $\sigma$ by $g^{\wedge k}_{\sigma*}$ and $g^{\wedge k}_{*\tau}$, respectively.
For $S \subseteq \binom{[n]}{k}$, we write $g^{\wedge S}$ for the $\abs{S} \times \binom{n}{k}$-submatrix of $g^{\wedge k}$ with rows indexed by $S$.

\begin{definition}\label{def:shift}
	The \emph{partial shift} of $S \subseteq \binom{[n]}{k}$ by $g \in \GL(n, \EE)$ is
	\[
		\Delta_g(S) \coloneqq \Set[\big]{ \sigma \in \tbinom{[n]}{k}; g^{\wedge S}_{*\sigma} \notin \Span_{\EE}[\big]{(g^{\wedge S}_{*\rho})_{\rho \lex< \sigma}} }.
	\]
\end{definition}

We call $g$ \emph{generic} if all entries of $g$ are algebraically independent over $\FF$.
In this case, $\Delta(S)\coloneqq \Delta_g(S)$ is shifted (called the \emph{full shift} of $S$ by $g$) and does not depend on $g$; see~\cite[Theorem 2.1]{Kalai90}.
Let $\EE = \FF(x_{ij} \mid 1 \leq i, j \leq n) \supset \FF$ be a field extension where the $x_{ij}$ are algebraically independent over $\FF$.
Then the matrix $\frX = (x_{ij})_{ij}$ is generic, so $\Delta(S) = \Delta_{\frX}(S)$.
Algebraic shifting can be considered with respect to any fixed total order on $\binom{[n]}{k}$ refining $\leq$; throughout we stick to $\lex\leq$.
\begin{remark}
	\label{rem:Monte Carlo}
	By~\cite[Theorem 2.1.6]{Kalai02} the full shift only depends on the characteristic of the field $\FF$.
	So it suffices to consider field extensions $\EE \supseteq \FF$ where $\FF$ is a prime field, and $\EE$ is the field of rational functions over $\FF$ with $n^2$ indeterminates.
	We omit the field from the notation $\Delta(S)$.
\end{remark}

\begin{remark}
	\label{thm:shift preserves cardinality}
	It is known that $\abs{\Delta_g(S)} = \abs{S}$; see~\cite[Theorem~3.1]{BjornerKalai:1988}.
	We stress that the genericity of $g$ is not essential for the proof.
\end{remark}

For $n \in \NN$, let $\SymmetricGroup{n}$ denote the symmetric group on $n$ letters.
For $w \in \SymmetricGroup{n}$, let
\[
	\inv w \coloneqq \Set{(i,j); i < j, i \cdot w > j \cdot w},
\]
be the set of inversions of $w$.
Define the $n \times n$-matrices $\frU(w)$ and $\frR(w)$
with entries in the field $\FF(x_{\inv w}) \coloneqq \FF(x_{ij} \mid (i,j) \in \inv w)$ by
\begin{equation}
	\label{eq:def u(w)}
	\frU(w) \coloneqq
	\begin{psmallmatrix}
		1 & u_{12}     & \cdots     & u_{1n}               \\
		  & \ddots[15] & \ddots[15] & \vdots               \\
		  &            & 1          & \!\! u_{n-1,n}\!\!\! \\
		  &            &            & 1
	\end{psmallmatrix}
	\ \text{with} \
	u_{ij} = \begin{cases*}
		x_{ij} & if $(i,j) \in \inv w$, \\
		0      & otherwise
	\end{cases*}
\end{equation}
and $\frR(r) \coloneqq \frU(w)w$.
\begin{definition}\label{def:shift:w}
  The \emph{partial shift} of $S \subseteq \binom{[n]}{k}$ by $w \in \SymmetricGroup{n}$ is
  $\Delta_{\frR(w)}(S)$.
\end{definition}
The full shift of a hypergraph is recovered as the partial shift with respect to the longest word $w_0$ of $\SymmetricGroup{n}$, seen as the Coxeter group of type A$_{n-1}$.

\begin{proposition}[{\cite[Proposition~11]{VecchiaJoswigLenzen:2024}}] 
  \label{thm:arxiv}
  For $S \subseteq \binom{[n]}{k}$ we have $\Delta(S) = \Delta_{\frR(w_0)}(S)$.
\end{proposition}
In other words, one can use the matrix $\frR(w_0)$ (which has $\frac{1}{2}n(n-1)$ generic entries)
instead of $\frX$ (which has $n^2$ generic entries) to compute $\Delta(S)$.
The matrix $\frR(w)$ contains the least number of generic entries necessary to compute the partial shift $\Delta_{\frR(w)}(S)$ for any $S$.
For a given $S$, it is possible to compute $\Delta_{\frR(w)}(S)$ with a matrix with even fewer generic entries,
as shown in the following lemma and the corollary below.
For $\sigma = \{i_1 < \dotsb < i_k\}$, let $\sgn(i_k,\sigma) = (-1)^i$.

\begin{lemma}
  \label{thm:facet substitution}
  Let $b \in \GL(n, \EE)$ be a matrix such that $b_{ij} = 0$ unless $\RepVert{\sigma}{i}{j} \in S$ for all $\sigma \in S$.
  Then $\Delta_{bg}(S) = \Delta_g(S)$ for all $g \in \GL(n, \EE)$.
\end{lemma}
\begin{proof}
  The assumption implies that $b$ must be a product of the elementary matrices
  \[
  	d^u(\lambda) = \begin{pNiceMatrix}[small,first-col,first-row]
  		  &   &        & u       &        &   \\
  		  & 1 &        &         &        &   \\
  		  &   & \ddots &         &        &   \\
  		u &   &        & \lambda &        &   \\
  		  &   &        &         & \ddots &   \\
  		  &   &        &         &        & 1
  	\end{pNiceMatrix},\
  	t^{ij}(\lambda) = \begin{pNiceMatrix}[small,first-col,first-row]
  		  &   &        & i       &        & j        &        &   \\
  		  & 1 &        &         &        &          &        &   \\
  		  &   & \ddots &         &        &          &        &   \\
  		i &   &        & 0       &        & \lambda  &        &   \\
  		  &   &        &         & \ddots &          &        &   \\
  		j &   &        & \lambda &        & 0        &        &   \\
  		  &   &        &         &        &          & \ddots &   \\
  		  &   &        &         &        &          &        & 1
  	\end{pNiceMatrix},\
  	l^{ij}(\lambda) = \begin{pNiceMatrix}[small,first-col,first-row]
  		  &   &        & i &        & j        &        &   \\
  		  & 1 &        &   &        &          &        &   \\
  		  &   & \ddots &   &        &          &        &   \\
  		i &   &        & 1 &        & \lambda  &        &   \\
  		  &   &        &   & \ddots &          &        &   \\
  		j &   &        &   &        & 1        &        &   \\
  		  &   &        &   &        &          & \ddots &   \\
  		  &   &        &   &        &          &        & 1
  	\end{pNiceMatrix}
  \]
  with $\RepVert{\sigma}{i}{j} \in S$ for all $\sigma \in S$.
  Hence, it suffices to show the statement for $b \in \{d^u(\lambda)$, $t^{ij}(1)$ and $l^{ij}(\lambda)\}$.
  We have
  \begin{align*}
  	(d^u(\lambda))^{\wedge k}    & = \textstyle \prod_{i \in \sigma \in \binom{[n]}{k}} d^\sigma(\lambda)                                                                                          \\
  	(t^{ij}(1))^{\wedge k}       & = \textstyle \prod_{i \in \sigma \in \binom{[n]}{k}} t^{\sigma, \RepVert{\sigma}{i}{j}}\bigl(\sgn(i,\sigma)\sgn(j,\RepVert{\sigma}{i}{j})\bigr) \\
  	(l^{ij}(\lambda))^{\wedge k} & = \textstyle \prod_{i \in \sigma \in \binom{[n]}{k}} l^{\sigma, \RepVert{\sigma}{i}{j}}(\lambda).
  \end{align*}
	We see that $(d^u(\lambda))^{\wedge k} \in \GL(\binom{n}{k},\EE)$ is diagonal.
	From the assumption on $i$, $j$, we get that $(t^{ij}(1))^{\wedge k}$ and $(l^{ij}(\lambda))^{\wedge k}$ have all its non-zero off-diagonal entries in rows and columns
	corresponding to $\sigma \in S$.
	This ensures that the second equality in
  \[
  	(bg)^{\wedge S} = b^{\wedge S} g^{\wedge k} = (b^{\wedge k})_{SS} \; g^{\wedge S}
 	\]
	holds.
	Hence, the columns of $(bg)^{\wedge S}$ have the same linear relations as those of $g^{\wedge S}$.
\end{proof}

For $v \in \EE^{\inv w}$, we write $\frU(w)(v)$ for the matrix obtained by putting in $v_{ij}$ for $x_{ij}$ in $\frU(w)$ for each $(i,j) \in \inv w$,
and analogously for $\frR(w)(v)$.
We obtain the following corollary from \cref{thm:facet substitution}:

\begin{corollary}
	\label{thm:left-invariance}
	Let $S \subseteq \binom{[n]}{k}$ and $w \in \SymmetricGroup{n}$.
	Let $v \in \EE^{\inv w}$ have entries $v_{ij} = 0$ if $\RepVert{\sigma}{i}{j} \in S$ for all $\sigma \in S$,
	and $v_{ij} = x_{ij}$ otherwise.
	Then we have $\Delta_{\frU(w)}(S) = \Delta_{\frU(w)(v)}(S)$. 
\end{corollary}
\begin{proof}
	We first show that for any unipotent upper triangular matrix $b \in \GL(n,\EE)$, 
	the matrix $\tilde{b}$ with entries
	\[
          \tilde{b}_{ij} = \begin{cases}
		0 & \text{if $\RepVert{\sigma}{j}{i} \in S$ for all $\sigma \in S$,}\\
		b_{ij} & \text{otherwise}
                           \end{cases}
        \]
	satisfies $\Delta_{\tilde{b} g}(S) = \Delta_{bg}(S)$ for all $g\in\GL(n,\EE)$.
	To see this, note that \cref{thm:facet substitution} implies that
	\[
		\tilde{b} = l^{12}(-b_{12}) \bigr[l^{13}(-b_{13}) l^{23}(-b_{23})\bigr] \dotsm \bigl[l^{1n}(-b_{1n})\dotsm l^{n-1,n}(-b_{n-1,n})\bigr] b.
	\]
	Now, to prove the statement, use this with $g = w$, $b = \frU(w)$ and $\tilde{b} = \frU(w)(v)$.
\end{proof}

In terms of examples we are mostly interested in a topological setting, where hypergraphs occur in the following way.
An \emph{abstract simplicial complex} $K$ on the vertex set $[n]$ is a nonempty set of nonempty subsets of $[n]$ which is closed with respect to taking subsets.
An element of $K$ is called a \emph{face}, and its dimension is the cardinality minus one.
The dimension of $K$ is the maximal dimension of its faces.
The faces of fixed dimension $q$ form a $(q{+}1)$-uniform hypergraph $K^{(q)}$.
A simplicial complex is \emph{shifted} if all those hypergraphs are shifted.
Conversely, for every $g \in \GL(n, \EE)$, the partially shifted uniform hypergraphs $\Delta_g(K^{(q)})$ for $q=0,\dotsc,\dim K$ again form a simplicial complex, denoted by $\Delta_g(K)$ \cite[288]{BjornerKalai:1988}%
\footnote{Scrutinizing the proof in \cite{BjornerKalai:1988} reveals that the standing assumption that the matrix $g$ is generic is never used.}.
As for hypergraphs, we write $\Delta(K)$ for $\Delta_g(K)$ if $g$ is generic.

\section{Shifting Algorithms} 
\label{sec:shift-algo}

\begin{algorithm}[tbp]
	\caption{Computing $\Delta_{\frR(w)}(S)$ naively. Choosing $w = w_0$ yields $\Delta(S)$.}
	\label{alg:naive shift}
	\KwIn{$S \subseteq \binom{[n]}{k}$, $w \in \SymmetricGroup{n}$}
	\KwOut{$\Delta_{\frR(w)}(S)$}
	$m \gets {}$any row echelon form of $\frR(w)^{\wedge S}$\tcp*{matrix over $\FF(x_{\inv w})$}
	\Return $\Set{\sigma \in \binom{[n]}{k}; \text{$m_{*\sigma}$ contains a pivot}}$
\end{algorithm}

This section contains our most relevant algorithmic contributions.
Additionally, at the end we summarize other approaches, which are dedicated to special cases.
Our algorithms usually operate on uniform hypergraphs; however, \cref{algo:simplicial complex} exploits the stucture of a simplicial complex.

Of course, $\Delta(S)$ and $\Delta_{\frR(w)}(S)$ can be determined directly, by computing the row echelon form $m$ of $\frR(w)^{\wedge S}$, where $\frR(w) \in \GL(n,\EE)$ for $\EE = \FF(x_{\inv w})$.
From that row echelon form the hyperedges of $\Delta_{\frR(w)}(S)$ correspond to those columns of $m$ containing pivots; see \cref{alg:naive shift}.
However, calculating in this field of fractions is very expensive, whence experimenting with smaller fields and avoiding arithmetic operations are natural ideas.
One way of computing (full) shifts in characteristic zero, i.e., $\FF=\QQ$, is to shift with respect to a random matrix in $\EE=\RR$.
Such a matrix is generic almost surely.
This leads to the standard Monte Carlo algorithm, implemented, e.g., in \cite{M2ExtShifting}.
That method has two key disadvantages.
First, there is no way to certify the correctness of the output.
Secondly, this idea does not work if $\Char \FF$ is positive.
We address both issues.

\subsection{Verifying partial shifts}
Our first goal is a general procedure for verifying for a given hypergraph $S$ and matrix $u$ if $\Delta_{uw_0}(S) = \Delta_{\frR(w_0)}(S)$; this is \cref{algo:verify}.
We will then use that method to devise a Las Vegas algorithm for computing $\Delta(S)$; see \ref{sec:lv} below.
While our procedure does not provide an advantage in the worst case, in terms of practical computations the difference can be decisive.
The practical advantage arises from the fact that, in many cases, the number of necessary computations in $\EE$ is much lower than for computing $\Delta(S) = \Delta_{\frR(w_0)}(S)$ directly.

Actually, the method even works for arbitrary partial shifts $\Delta_{\frR(w)}(S)$ for $w \in \SymmetricGroup{n}$.
Therefore, it is natural to phrase the algorithm in a way where the input consists of a hypergraph $S$, a permutation $w$, and the candidate matrix $u$.
Additionally, \cref{algo:verify} accepts a list $D$ of hyperedges in $\binom{[n]}{k}$ known not to belong to $\Delta_{\frR(w)}(S)$, which will be skipped during the algorithm.
This will be used later when computing shifts of simplicial complexes; see \cref{algo:simplicial complex}.

\newcommand{\Ind}{\operatorname{ind}}
For a matrix $g$ and uniform hypergraphs $S, T \subseteq \binom{[n]}{k}$,
let $g^{\wedge S}_{*T} \coloneqq (g^{\wedge S}_{*\tau})_{\tau \in T}$.
For an $l \times \binom{n}{k}$-matrix $m$, let $\Ind m \coloneqq \Set{\sigma \in \binom{[n]}{k}; m_{*\sigma} \notin \Span{m_{*\rho}; \rho \lex< \sigma}}$.
Note that if $m$ is in row echelon form, one can read off $\Ind m$ directly.
The idea behind \cref{algo:verify} below is to compute $\Delta_{uw}(S)$ for a matrix $u \in \EE^{n\times n}$, which only involves computing a row echelon form over $\EE$ rather than $\FF(x_{\inv w})$,
and then to verify if $\Delta_{uw}(S) = \Delta_{\frR(w)}(S)$ without computing the entire row echelon form of $\frR(w)^{\wedge S}$.
\begin{algorithm}[tb]
	\caption{Verify if $u$ computes the partial shift of a uniform hypergraph $S$ w.r.t.\ some field extension $\EE \supseteq \FF$}
	\label{algo:verify}
	\LinesNumbered
	\KwIn{$S \subseteq \binom{[n]}{k}$, $w \in \SymmetricGroup{n}$, $u \in \EE^{n{\times}n}$ upper triangular unipotent,
        $D \subseteq \binom{[n]}{k} \setminus \Delta_{\frR(w)}(S)$ a possibly empty set of known non-faces.}
	\KwOut{\True\ if $\Delta_{\frR(w)}(S) = \Delta_{uw}(S)$; otherwise \False.}
	$S' \gets \Delta_{uw}(S)$\;
	\lIf{$w = w_0$ and $S'$ is not shifted}{\Return \False}                                              \label{algo:verify:exit-not-shifted}
  $T \gets \Set{\tau \in \binom{[n]}{k} ; \tau \lex\leq \lex\max S'} \setminus (S' \cup D)$\;          \label{algo:verify:T}%
  $C \gets \Set{\tau \in \binom{[n]}{k} ; \tau \lex\leq \lex\max T}$\;
  $r \gets \frR(w)$\;                                                                                  \label{algo:verify:r}
  \lForEach{$(i,j) \in \inv w$ with $\RepVert{\sigma}{j}{i}$ for all $\sigma \in S$}{$r_{ij} \gets 0$} \label{algo:verify:r-loop}
  $m \gets (r^{\wedge S})_{*C}$ \tcp*{$\FF(x_{\inv w})$-matrix}
  $m_{*,D} \gets 0$\;                                                                                  \label{algo:verify:m-columns}
  \lIf{$w = w_0$}{%
     	$m_{*,(T\setminus\min T)} \gets 0$ \tcp*[f]{\smash{\begin{varwidth}[t]{4cm}\raggedleft set of minima w.r.t.\ domination order\end{varwidth}}}
  }                                                                                                    \label{algo:verify:horizontal}
	\ForEach(\label{algo:verify:loop}){$\sigma \in T$ in lex order}{%
    \lIf{$ m_{*\sigma } \notin  \Span{m_{*\rho}; \rho \lex< \sigma}$ }{\Return \False}                 \label{algo:verify:exit}
	}
	\Return \True
\end{algorithm}
\begin{theorem}
  \label{thm:algo:a}
  Given $S \subseteq \binom{[n]}{k}$, $w \in \SymmetricGroup{n}$ and a unipotent matrix $u \in \GL(n, \EE)$,
  \cref{algo:verify} correctly decides if $\Delta_{\frR(w)}(S) = \Delta_{uw}(S)$.
\end{theorem}

The version of \cref{algo:verify} given here is a substantial improvement over the version from our ISSAC extended abstract.
In particular, we introduce the set $D \subseteq \binom{[n]}{k} \setminus \Delta_{\frR(w)}(S)$ of known non-faces as an additional parameter.
This is useful for computing shifts of simplicial complexes which are computed as a sequence of shifts of the hypergraphs formed by the faces of a fixed dimension.
By proceeding from lower to higher dimensions the information about non-faces can be passed on, for a considerable speed gain in practice; see \cref{sec:balls-and-spheres}.
To prove the above \lcnamecref{thm:algo:a}, we need the following.

\begin{lemma}
  \label{thm:concrete-lex-less}
  For every $S \subseteq \binom{[n]}{k}$, $w \in \SymmetricGroup{n}$ and $v \in \EE^{\inv w}$, we have
  $\Delta_{\frR(w)}(S) \lex\leq \Delta_{\frR(w)(v)}(S)$.
\end{lemma}
\begin{proof}
	We prove the statement inductively.
	For brevity, we write $\frR \coloneqq \frR(w)$ and $\tilde{\frR} \coloneqq \frR(w)(v)$.
	We have to show that $\Ind \frR \lex\leq \Ind \tilde{\frR}$.
	Assume that up to and excluding the $\sigma$th column, we have $\Ind \frR_{*,<\sigma} = \Ind \tilde{\frR}_{*,<\sigma}$.
	If $\frR_{*\sigma} \in \Span{\frR_{*,<\sigma}}$, then also $\tilde{\frR}_{*\sigma} \in \Span{\tilde{\frR}_{*,<\sigma}}$.
	Hence, $\Ind \frR_{*,\leq\sigma} = \Ind \frR_{*,<\sigma} = \Ind \tilde{\frR}_{*,<\sigma} = \Ind \tilde{\frR}_{*,\leq\sigma}$.
	If $\frR_{*\sigma} \notin \Span{\frR_{*,<\sigma}}$, we have $\Ind \frR_{*,\leq\sigma} = \Ind (\frR_{*,<\sigma}) \cup \{\sigma\}$,
	so $\Ind \frR \lex\leq \Ind\tilde{\frR}$.
\end{proof}

\begin{proof}[Proof of \cref{thm:algo:a}]
	Note that we have $uw = \frR(w)(v)$ for $v = (u_{ij})_{ij \in \inv w}$.
	Then \cref{thm:concrete-lex-less} implies that $\Delta_{\frR(w)}(S) \lex\leq \Delta_{uw}(S)$.
	Therefore, we know that each column $(\frR(w))^{\wedge S}_{*\tau}$ for $\tau \lex> \lex\max S'$ must lie in the span of columns left of it.
	In other words, no such column can contain a step in the row echelon form $m$.
	We may thus safely discard all such columns from $\frR(w)^{\wedge S}$.

  The columns of $\frR(w)(v)^{\wedge S}_{*S'}$ are linearly independent by definition of $S'$.
  Hence, also the columns of $\frR(w)^{\wedge S}_{*S'}$ must be linearly independent.
  To show that $S' = \Delta_{\frR(w)}(S)$, that is that $\Ind\frR(w)(v)^{\wedge S} = \Ind\frR(w)^{\wedge S}$,
  we need only show that each column in $\frR(w)^{\wedge S}$ corresponding to a $\tau \in T$ for $T = \Set{\tau \in \binom{[n]}{k} ; \tau \lex\leq \lex\max S'} \setminus (S' \cup D)$ lies in the span of the columns to the left.
  This is the step in line~\ref{algo:verify:exit}.
  The lines \ref{algo:verify:r}--\ref{algo:verify:m-columns} build the matrix $m$ and set certain entries or columns of it to zero.
  In the following, we justify that these changes do not alter the result.
  
  First, by \cref{thm:left-invariance}, setting entries of $\frR(w)$ to zero as done in line~\ref{algo:verify:r-loop}
  does not change the result.
  Secondly, taking for granted that $D$ does not contain hyperedges from $\Delta_{\frR(w)}(S)$,
  we know that $m_{*\tau}$ is contained in the span of columns left of it for every $\tau \in D$.
  Setting these columns to zero thus does not change the result.
  
  Lastly, the case $w = w_0$ allows for further improvements.
  Because we know that $\Delta_{\frR(w)}(S)$ is shifted,
  the algorithm may exit already in line~\ref{algo:verify:exit-not-shifted} if $S'$ is not shifted.
  Further, we can avoid some reductions for deciding line~\ref{algo:verify:exit} by removing all columns $m_{*\tau'}$ for $\tau' \geq \tau \in T$.
  That is, if we have shown that $\tau$ lies in the span of the columns left of it (i.e., $\tau$ is a non-face of the shifted complex),
  we already know from shiftedness that the same is true for every $\tau' > \tau$.
  Consequently, setting those columns $\tau'$ to zero does not change the result of the computation.
  Phrased differently, the verification at line~\ref{algo:verify:exit} needs to be done only for $\tau \in \min T$, where $\min T$ is the set of minima with respect to the domination order,
  and we can safely set all columns of $m$ corresponding to non-minima of $T$ to zero, which justifies line~\ref{algo:verify:horizontal}.
\end{proof}

\begin{remark}
	We implement the step in \cref{algo:verify}, line~\ref{algo:verify:exit} by computing a row echelon form of $m$.
	Note that if $T = \emptyset$, then $m$ has no columns, so the row echelon form of $m$ is trivial.
	This can leads to drastic speedups; see \cref{sec:surfaces}.
	We stress that \cref{algo:verify} is transparent with respect to how line~\ref{algo:verify:exit} is decided.
	An approach using a different algorithm for computing row echelon forms is described and evaluated in \cref{sec:fabian's reduction,subsubsec:lazy}.
	Investigating the interaction of our shifting algorithm with the various heuristics involved in computing row echelon forms over fraction fields
	is beyond the capacities of this paper;
	but see \cref{rmk:increased-time} for an example where restricting to a \emph{subset} of columns of a matrix \emph{increases} the time needed to compute the row echelon form.
\end{remark}

\begin{remark}
  The definition of the matrix $\frR(w)$ in \cref{algo:verify} implicitly uses the fixed elements $\frX$ from \cref{eq:def u(w)}.
\end{remark}

\begin{remark}
  A variant of \cref{algo:verify} can be used to decide if a given hypergraph $U$ is the $w$-shift of another hypergraph $S$.
  Namely, given $U, S \subseteq \binom{[n]}{k}$ and $w \in \SymmetricGroup{n}$ one can check if $U = \Delta_{\frR(w)}(S)$
  by an algorithm obtained from \cref{algo:verify} by replacing ``$\setminus (S' \cup D)$''  in line \ref{algo:verify:T} with ``$\setminus D$''.
  To motivate this, observe that we cannot invoke \cref{thm:concrete-lex-less} in this setting.
  Consequently, in contrast to \cref{algo:verify:exit} for $w=w_0$, we need to check more columns for deciding if there exists some matrix $u$ such that $U = \Delta_{uw}(S)$.
\end{remark}

\usetikzlibrary{calc,decorations.pathmorphing,shapes}

\newcommand\xleadsto[2][]{%
	\mathrel{%
		\begin{tikzpicture}[baseline={($(sup.south) + (0,-0.5ex) $)}]
			\begin{scope}[local bounding box=b]
				\node[inner sep=.5ex] (sup) {$\scriptstyle #2$};
				\node[inner sep=.5ex, anchor=north] (sub) at (sup.south) {$\scriptstyle #1$};
			\end{scope}
			\coordinate (n1) at (sup.south -| b.west);
			\coordinate (n2) at (sup.south -| b.east);
			\path[draw,<-,decorate, decoration={snake,amplitude=0.9pt,segment length=2mm,pre=lineto,pre length=4pt}] ($(n2)+(.5ex,0)$) -- (n1);
		\end{tikzpicture}%
	}%
}

We study our \cref{algo:verify} via a pair of examples.
To simplify the notation, in the sequel, we write $i_1\cdots i_k$ as shorthand for the hyperedge  $\{i_1, \dotsc, i_k\}$.
\begin{example}
	\label{ex:algo:a}
	\colorlet{omittable}{red} 
	Let $S = \{13, 14, 23, 24\} \in \binom{[4]}{2}$, $w = (1 \enspace 2 \enspace 3 \enspace 4)$, $\EE=\FF=\GF{2}$, and
	\[
		u = \begin{psmallmatrix}
			1 & 0 & 0 & 0 \\
			0 & 1 & 0 & 0 \\
			0 & 0 & 1 & 1 \\
			0 & 0 & 0 & 1
		\end{psmallmatrix},
		\qquad
		u' = \begin{psmallmatrix}
			1 & 0 & 0 & 0 \\
			0 & 1 & 0 & 1 \\
			0 & 0 & 1 & 1 \\
			0 & 0 & 0 & 1
		\end{psmallmatrix}.
	\]
	For the matrix $u$, \cref{algo:verify} gives $S' = \Delta_{uw}(S) = \Set{12, 13, 14, 34}$,
	$T = \{23, 24\}$, and $C = \{12, 13, 14, 23, 24\}$.
	Let us ignore line~\ref{algo:verify:r-loop} of \cref{algo:verify} for a moment.
	From the row echelon form
	\begin{equation}
		\label{eq:ex:algo:a:1}
		\begin{psmallmatrix}
			x_{34} & 0      & x_{14} & 0 & 1      \\
			0      & x_{34} & x_{24} & 0 & 0      \\
			0      & 0      & x_{14} & 0 & 1      \\
			0      & 0      & 0      & 0 & x_{24}
		\end{psmallmatrix}
		\quad\text{of}\quad
		m = \frR(w)^{\wedge S}_{*C} = \begin{psmallmatrix}
			\underline{x_{34}} & 0                  & x_{14} & 0 & 1 \\
			1                  & 0                  & 0      & 0 & 0 \\
			0                  & \underline{x_{34}} & x_{24} & 0 & 0 \\
			0                  & 1                  & 0      & 0 & 0
		\end{psmallmatrix},
	\end{equation}
	we see that \cref{algo:verify} returns \False\ in line~\ref{algo:verify:exit} for $\sigma = 24$.
	On the other hand, for the matrix $u'$, the algorithm computes
	$S' = \Delta_{u'w}(S) = \Set{12, 13, 14, 24}$, $T = \Set{23}$ and $C = \Set{12, 13, 14, 23}$.
	From the row echelon form
	\begin{equation}
		\label{eq:ex:algo:a:2}
		\begin{psmallmatrix}
			x_{34} & 0      & x_{14} & 0 \\
			0      & x_{34} & x_{24} & 0 \\
			0      & 0      & x_{14} & 0 \\
			0      & 0      & 0      & 0
		\end{psmallmatrix}
		\quad\text{of}\quad
		m = \frR(w)^{\wedge S}_{*C} = \begin{psmallmatrix}
			\underline{x_{34}} & 0                  & x_{14} & 0 \\
			1                  & 0                  & 0      & 0 \\
			0                  & \underline{x_{34}} & x_{24} & 0 \\
			0                  & 1                  & 0      & 0
		\end{psmallmatrix},
	\end{equation}
	we see that \cref{algo:verify} returns \True.
	Indeed, $S' = \Delta_{\frR(w)}(S)$.
	Lastly, recall that with knowledge of $S$, \cref{thm:left-invariance} allows to set certain entries of $m$ to zero, as done by line~\ref{algo:verify:r-loop}.
	In the situation of this example, we can drop the underlined entries $x_{ij}$ in $\frR(w)$ without changing the result (of course, the row echelon form will change, but not the set $\Ind m$).
\end{example}

\subsection{A Las Vegas algorithm for partial shifting}
\label{sec:lv}
\begin{algorithm}[bp]
	\caption{Las Vegas algorithm for computing partial shifts of uniform hypergraphs over some field extension $\EE \supseteq \FF$}
	\label{algo:lv}
	\KwIn{$S\subseteq\binom{[n]}{k}$, $w \in \SymmetricGroup{n}$}
	\KwOut{$\Delta_{\frR(w)}(S)$}
	\DoUntil{\cref{algo:verify} returns \True\enspace on $S$, $u$ and $w$}{
		\nlset{$(*)$}\label{algo:lv:u}
		$u \gets {}$matrix with $u_{ij} =
		\smash[t]{
			\begin{cases}
				\text{random value in $\EE$} & \text{if $(i,j) \in \inv w$,}\\
				1 & \text{if $i = j$,}\\
				0 & \text{otherwise}
			\end{cases}
		}$
	}
	\Return $\Delta_{uw}(S)$
\end{algorithm}
\Cref{algo:verify} lends itself readily to the Las Vegas algorithm for computing $\Delta_{\frR(w)}(S)$ listed in \cref{algo:lv}.
	Note that if $\FF = \QQ$ and $\EE = \RR$, then the $u_{ij}$ picked in \cref{algo:lv}, line~\ref{algo:lv:u}
	will be algebraically independent over $\FF$ almost surely.
	In particular, \cref{algo:lv} terminates almost surely in this case.

Of course, the Las Vegas algorithm does not need to terminate in general.
In particular, the field $\EE$ may be chosen too small to fit a matrix which is generic enough.
This can be observed in the following example.

\begin{example}
	\label{ex:not every generic shift is concrete}
	Let $S = \{124, 125, 236, 245, 256, 345, 456\}$, $\FF = \GF{2}$.
	Then $\Delta(S) = \Delta_{\frR(w_0)}(S) = \{123, 124, 125, 126, 134, 135, 136\}$, but no $g \in \GL(n, \GF{2})$ satisfies $\Delta(S) = \Delta_g(S)$.
	Over $\GF{4}$, in contrast, such $g$ do exist; for example,
	\[
		\begin{psmallmatrix}
			0 & 0 & 0 & 1 & \alpha+1 & 1 \\
			1 & \alpha & 0 & 0 & 1   & 0 \\
			0 & 0 & 0 & 1 & 0   & 0 \\
			0 & 0 & 1 & 0 & 0   & 0 \\
			1 & 1 & 0 & 0 & 0   & 0 \\
			1 & 0 & 0 & 0 & 0   & 0
		\end{psmallmatrix}
	\]
	where $\GF{4} = \GF{2}[\alpha]/(\alpha^2+1)$.
	This example is has the smallest dimension, number of vertices and number of hyperedges with this property.
	
	The example goes beyond our ISSAC extended abstract, which provides a simpler example for a hypergraph $S$ such that no $\GF{2}$-matrix $g$ satisfies $\Delta_{gw}(S) = \Delta_w(S)$ for a certain $w < w_0$.
	The example has been found using an exhaustive search with our implementation in the \OSCAR\ computer algebra system \cite{Oscar,OSCAR-book}.
	Of course, one has to use the deterministic \cref{alg:naive shift}, rather than the Las Vegas algorithm~\ref{algo:lv},
	to find and confirm such an example.
	We note that using a naive implementation using the matrix $\frX$, rather than $\frR(w_0)$,
	it goes beyond realistic computational capacities to find this example.
\end{example}

\subsection{Shifting simplicial complexes}
So far, we have only dealt with algebraic shifting of uniform hypergraphs.
If $K$ is a simplicial complex on $n$ vertices, then the set $K^{(q)}$ of $q$-dimensional simplices of $K$ is a $(q+1)$-uniform hypergraph for every $q$.
For every $g \in \GL(n,\EE)$, their shifts $\Delta_g(K^{(q)})$ form a simplicial complex again, denoted by $\Delta_g(K)$ \cite[Propsition~35]{VecchiaJoswigLenzen:2024}.
Therefore, $\Delta_g(K)$ is determined by the $\Delta_g(K^{(q)})$ for the $K^{(q)}$ that contain at least one facet of $K$.

When computing the partial shift of a simplicial complex $K$ we may take advantage of the fact that $\Delta_g(K)$ is a simplicial complex.
Since the shifting of each $K^{(q)}$ is done independently, we may choose the order in which we compute the partial shifts.
Namely, computing $\Delta_g(K^{(q)})$ before $\Delta_g(K^{(q+1)})$, we know that every coface $\tau$ of non-face $\sigma \in \binom{[n]}{q+1} \setminus \Delta_g(K^{(q)})$ is a non-face of $\Delta_g(K^{(q+1)})$.
Thus, we do not have to do any calculations in \cref{algo:verify} for the column corresponding to $\tau$.
The set of known such non-faces is passed to \cref{algo:verify} via the input $D$;
see \cref{algo:simplicial complex}, which can be used, mutatis mutandis, 
together with \cref{algo:lv} for obtaining a Las Vegas algorithm for computing partial shifts of simplicial complexes.

It is then natural to ask for which simplicial complexes $K$ and which $K^{(q)}$ can we avoid the loop at \cref{algo:verify}, line~\ref{algo:verify:loop} altogether.
Before we begin to answer this question we will need to understand some basic properties of near cones~\cites[\S4]{BjornerKalai:1988}{Nevo:2005} and their relationship to algebraic shifting.
Recall the notation $\RepVert{\sigma}{i}{j}$ from \eqref{eq:replace-vertex}.

\begin{algorithm}[bp]
  \caption{Verify if $u$ computes the partial shift of a simplicial complex w.r.t a field $\EE \supseteq \FF$}
  \label{algo:simplicial complex}
  \KwIn{$K$ a simplicial complex on $n$ vertices of dimension $d$, $w \in \SymmetricGroup{n}$}
  \KwOut{\True\ if $\Delta_{\frR(w)}(K) = \Delta_{uw}(K)$; otherwise \False.}
  $D \gets \emptyset$\;
  \For{$q = 1, \dotsc, d$}{
    \If{\cref{algo:verify} returns \False \enspace for $K^{(q)}, u, w$ and $D$}{\Return \False}
    $S \gets \Delta_{uw}(K^{(q + 1)})$\;
    $\sigma_{\max} \gets \lex\max S$\;
    $D \gets \Set{\tau \in \binom{[n]}{k + 1} ; \tau \lex\leq \sigma_{\max}, \tau \notin S, \partial \tau \not \subseteq \Delta_{uw}(K^{(q)})}$\;
  }
  \Return \True
\end{algorithm}

\begin{definition}
  \label{def:near cone}
  A simplicial complex $K$ is a \emph{near cone} if for every $\sigma \in K$ with $1 \notin \sigma$ we have $\RepVert{\sigma}{i}{1} \in K$
  for each $i \in \sigma$.
\end{definition}
Note that every shifted complex is a near cone, and in particular satisfies the following statement.

\begin{proposition}[{\cite[Theorem~4.3]{BjornerKalai:1988}}]
  \label{thm:betti near cone}
  Let $K$ be a near cone.
  Then $K$ is homotopy equivalent to a wedge of spheres and has Betti numbers $\beta_q(K) = \abs{\Set{\sigma \in K^{(q)} ; \sigma \cup \{1\} \notin K}}$.
\end{proposition}

We see that the non-faces $\sigma \ni 1$ of dimension $q$ count the $(q-1)$st Betti numbers.
Suppose that $K$ is a simplicial complex with $\beta_i = 0$ for $1 \leq i < \dim K$, then by \cref{thm:betti near cone} we know that any $q$-dimensional non-face $\sigma \ni 1$ must contain a lower dimensional non-face.
In the cases where $q \geq 2$, the set $T$ from step \ref{algo:verify:T} in \cref{algo:verify} will be empty if for any non-face $\sigma \not \owns 1$ we have $\sigma > \sigma_{\max}$, here $\sigma_{\max}$ denotes the lexicographically largest $q$-face of $\Delta(K^{(q)})$.
Putting things together we get the following statement.

\begin{proposition}
  \label{thm:complexity}
  Let $K$ be a simplicial complex of dimension $d$ with $\beta_i(K) = 0$ for $1 \leq i < d$, with $\Set{\sigma \in \Delta(K^{(q)}); 1 \notin \sigma}$ forming a $\lex\leq$-interval for each $q \geq 2$.
  Let $t$ denote the number of steps required to compute $\Delta(K^{(1)})$, then \cref{algo:simplicial complex} computes $\Delta(K)$ in $\mathcal{O}\bigl(t + \sum_{q =2}^d\bigl(\binom{n}{q + 1} - f_q\bigr) \bigl(\binom{n}{q} - f_{q - 1}\bigr)\bigr)$ steps, where $f$ denotes the $f$-vector of $K$.
\end{proposition}
\begin{proof}
  The above discussion shows that when $q > 1$ we need only check that non-faces $\sigma \ni 1$ of dimension $q$ contain lower dimensional non-faces.
  What remains is to bound the complexity of such checks.
  Note that an upper bound on the number of non-faces of dimension $q$ less than the lex largest face of $\Delta(K^{(q)})$ is $\binom{n}{q + 1} - f_q$.
  Checking whether $\sigma \in \binom{[n]}{q + 1}$ contains a lower dimensional non-face is linear in the number of $(q{-}1)$-dimensional non-faces, which is bounded by $\binom{n}{q} - f_{q - 1}$.
  Hence for a fixed dimension $q$, the number of steps required to check that all non-faces smaller than the lexicographically largest face of $\Delta(K^{(q)})$ is bounded by $\bigl(\binom{n}{q + 1} - f_q\bigr) \bigl(\binom{n}{q} - f_{q - 1}\bigr)$.
  Summing up over all $q \geq 2$ we obtain our bound.
\end{proof}

The authors of \cite{BulavkaNevoPeled25} have found that a typical exterior shift of a triangulated surface has the form of what they call a \emph{homology lex-segment complex}.
The previous result \cref{thm:complexity} shows that for a homology lex-segment complex with only top dimensional homology \cref{algo:simplicial complex} will only perform matrix reductions when computing the exterior shift of the $1$-skeleton.

\subsection{Other algorithms}
For the sake of completeness we briefly survey some special procedures for shifting.
\paragraph{Combinatorial shifting}
Apart from \emph{algebraic shifting} as introduced by Kalai~\cite{Kalai84,Kalai86}, there is also a notion of \emph{combinatorial shifting} due to Erd\H{o}s--Ko--Rado~\cite{ErdosKoEtAl:1961}.
The reader is referred to the surveys by Frankl~\cite{Frankl:1987} and Kalai~\cite[\S 6.2]{Kalai02} for the connection with extremal combinatorics.

\begin{definition}[{\cite[\S 2]{Frankl:1987}}]
	For $t \in \SymmetricGroup{n}$ a transposition, the \emph{combinatorial shift} $\CShift_t(S)$
	of a $k$-uniform hypergraph $S \subseteq \binom{[n]}{k}$
	is the $k$-uniform hypergraph $\CShift_t(S) \coloneqq \Set{\CShift_t(\sigma, S); \sigma \in S}$, where
	\[
		\CShift_t(\sigma, S) \ \coloneqq \ \begin{cases}
			\sigma \cdot t & \text{if $\sigma > \sigma \cdot t \notin S$}, \\
			\sigma         & \text{otherwise.}
		\end{cases}
	\]
	By construction we have $\abs{\CShift_t(S)} = \abs{S}$.
\end{definition}

It has been shown in \cite[Proposition~25]{VecchiaJoswigLenzen:2024} 
that $\Delta_{\frR(t)}(S) = \CShift_t(S)$
for all $S \subseteq \binom{[n]}{k}$ and all transpositions $t \in \SymmetricGroup{n}$.
We remark that in contrast to partial shifts by general permutations, $\CShift_t(S)$ can be computed in time $\mathcal{O}(\abs{S})$.
It is known, however, that not every exterior partial shift can be realized as a sequence of combinatorial shifts \cite[\S 6.2]{Kalai02}.

\paragraph{Shifting of low genus surface triangulations}
\label{sec:surfaces}
For the special case that $K$ is a triangulation of the two-torus, the real projective plane or the Klein bottle,
Keehn and Nevo~\cite{KeehnNevo:2024} demonstrated that $\Delta(K)$ can be computed from $K$ in polynomial time in the number of vertices, with degree depending on the topology of $K$.
Specifically, if $K$ has $n$ vertices, then their method computes computes $\Delta(K)$
in time $\mathcal{O}(n^8)$ if $K$ is a triangulation of the torus or the Klein bottle,
and in time $\mathcal{O}(n^5)$ if $K$ is a triangulation of the real projective plane.
Their method is restricted to exterior shifting in characteristic zero.

\section{Implementation details} 
\label{sec:implementation}

Finding the lexicographic minimum basis in \cref{algo:verify} uses Gaussian elimination over the fields $\EE$ and $\FF(x_{\inv w})$ to find a row echelon form.
For an $n\times n$-matrix, this requires $\mathcal{O}(n^3)$ arithmetic operations in the underlying field.
The interesting case is the field of functions $\FF(x_{\inv w})$ over a prime field $\FF$.

\subsection{Multivariate polynomials}
\OSCAR represents a multivariate polynomial as a hash table, with the multivariate exponents as the keys and the coefficients as the values.
Consequently, a multivariate rational function can be written as the pair of a numerator and a denominator which are coprime.
The linear algebra is usually organized to get away with arithmetic in the multivariate polynomial ring $\FF[x_{\inv w}]$ as much as possible.
Row reduction over such a ring uses the extended Euclidean algorithm, and then multiplying the rows involved in a row subtraction step by the gcd cofactors before the subtraction.
The pivot of a column is picked as the lowest index of a non-zero entry that minimizes the length (i.e., number of terms) of that entry.

\subsection{Finite fields}
As we put a special focus on computations in positive characteristic, it is worth-while to briefly discuss the implementation of finite fields.
\OSCAR has several versions of these.
This includes the type \texttt{FqField} for generic finite fields of arbitrary prime power order.
Our experiments employ the more efficient type \texttt{fpField}, which is restricted to fields of prime order, where the prime is a 64 bit unsigned integer.
The following detail is crucial for the efficacy of \cref{algo:verify}: even if $\EE = \GF{p^n}$ for $n > 1$ (implemented with \texttt{FqField}),
it suffices to compute line~\ref{algo:verify:exit} in the field $\FF = \GF{p}$ (implemented with \texttt{fpField}).
Calculating line~\ref{algo:verify:exit} over $\GF{p^n}$ would have a noticeable impact on the performance.

\subsection{Details concerning the Las Vegas algorithm}
\label{sec:lv-implementation-detail}
Recall that the Las Vegas \cref{algo:lv} first computes the shift $C = \Delta_g(S)$ for a random matrix $g$ with entries in $\FF$, where $\FF = \QQ$ or a finite field,
and then verifies that $C = \Delta(S)$ by a computation over $\EE = \FF(x_{\inv w})$.
We observed in prior experiments that in general, the first is orders of magnitudes faster than the second, due to the challenging computation over the fraction field $\EE$.

Consequently, we used the following variant of the Las Vegas Algorithm~\ref{algo:lv} for our experiments:
instead of line~\ref{algo:lv:u},
we pick random matrices $g_1,\dotsc,g_N$ as in \ref{algo:lv:u}, let $g = \lex\argmin\Set{\Delta_{g_i}(S); i \leq N}$, and then apply \cref{algo:verify} to $g$.
If not stated otherwise, we let $N = 500$ for $\EE$ finite, and $N = 1$ for $\EE = \QQ$.
The latter choice is motivated by the fact that in our experiments, we did not observe a single random rational matrix that was not generic enough.

\begin{remark}
	\label{rmk:ref-implementation}
	Recall that we decide \cref{algo:verify}, line~\ref{algo:verify:exit} by computing the row echelon form of the matrix $m$.
	This is done by the function \texttt{Oscar.ModStdQt.ref\_ff\_rc!} of \OSCAR.
	While the strategy and heuristics implemented in this function generally are suitable for our needs, we stress that precise strategy for computing row echelon forms
	can have drastic impact on the total running time and can lead to unexpected results (see also \cref{rmk:increased-time}).
\end{remark}

\subsection{Lazy row reduction}
\label{sec:fabian's reduction}
To underpin what we mean here, we propose the following alternative to computing the entire row echelon form.
Recall that neither for the deterministic nor the Las Vegas algorithm, we need the full row echelon form, but rather the positions of the pivots.
Therefore, we can spare some row reductions of elements above the pivots.
Namely, if a matrix $m$ has already been brought into the shape (via elementary row operations)
such that rows $1$ through $r$ are in row echelon form and each contain a step in columns $1$ through $s$,
then when considering the $(s{+}1)$st column of $m$, it is not necessary to evaluate the first $r$ entries of this column,
for they all lie in rows that already contain a step of the row echelon form.

Consequently, \cref{alg:lazy reduction} encodes the accumulated elementary row operations on $m$ by a helper matrix $v$,
and only evaluates the matrix-vector product $v m_{*,j}$ when deciding if the column $c = (vm)_{*j}$ contains a step.
In this case, there are different strategies for selecting the pivot of this column.
In our setting, where the entries of $c$ are polynomials,
we choose the least index $i$ of a non-zero entry $c_i$ that minimizes the length of $c_i$.

Note that \cref{alg:lazy reduction} computes the \emph{row} echelon form of $m$, but does so in a \emph{column} oriented way.

\begin{algorithm}[tb]
	\caption{Lazy row reduction. One can think of different strategies how to choose the pivot of a column in line~\ref{alg:lazy reduction:pivot}.}
	\label{alg:lazy reduction}
	\KwIn{$l\times n$-matrix $m$}
	\KwOut{$\Ind m$}
	$v \gets \text{$l\times l$-identity matrix}$\;
	$I \gets \emptyset$\;
	\For{$j=1,\dotsc,n$}{
		$c \gets (m_{\abs{I}+1,*},\dotsc,m_{l,*})^T v$\;
		\lIf{$c = 0$}{%
			\Continue
		}
		$I \gets I \cup \{j\}$\;
		\lIf{$\abs{I} = l$}{\Return $I$}
		\nlset{$(*)$}\label{alg:lazy reduction:pivot}%
		$i \gets \mathit{pivot}(c)$\;
		\If{$i \neq 1$}{
			swap entries $c_1$ and $c_i$\;
			swap rows $v_{\abs{I}+1,*}$ and $v_{\abs{I}+i,*}$\;
		}
		\For{$i = 2,\dotsc,m-\abs{I}$}{
			\lIf{$c_i = 0$}{\Continue}
			compute $a, b$ s.t.\ $ac_i + bc_1 = \gcd(c_i, c_1)$\;
			$v_{\abs{I}+i,*} \gets b v_{\abs{I}+i,*} - a v_{\abs{I}+1,*}$\;
			$v_{\abs{I}+i,*} \gets v_{\abs{I}+i,*} / \gcd(v_{\abs{I}+i,*})$\;
		}
	}
	\Return $I$\;
\end{algorithm}

\section{Computational experiments} 
\label{sec:experiments}

In this section, we report run time and memory consumption data for different algorithms for computing (partial) exterior shifting
of different uniform hypergraphs and simplicial complexes with coefficients in varying fields.
Our experiments were conducted with the computer algebra system \OSCAR, version 1.7.0 \cite{Oscar,OSCAR-book} on Julia 1.12.\footnote{%
  The timings in the ISSAC extended abstract were obtained using Julia 1.10 and \OSCAR 1.3.1.
}
To reproduce the results consult the GitHub repository \cite{BenchmarksRepo} for the details.

All experiments were run on a desktop computer running OpenSUSE Leap 15.6
with an AMD Ryzen 9 5900X 12-core processor and \SI{128}{\giga\byte} of main memory. 
Run time and memory consumption are listed as reported by the \texttt{@timed}-macro of Julia.
Note that the memory consumption estimates the total amount of allocated space during run time, not peak memory.
All instances were run in a separate Julia process, with a virtual memory limit set (via \texttt{ulimit -v}) to \SI{80}{\giga\byte},
and with a time limit set to \SI{30}{\minute} (excluding initialization and setup of the worker process); \texttt{ulimit} is a built-in Linux shell command.
If the instances hits the limit, the tables below contain entries \enquote{oom} (out of memory) and \enquote{oot} (out of time), respectively.

\subsection{Bipartite graphs}
\label{sec:bipartite}
Our first experiment is designed to compare the various algorithms for exterior shifting in characteristic zero.
Here we restrict our attention to exact algorithms that give provably correct results.
The Monte Carlo procedure is much faster (see \cref{subsec:surfaces} for timings with other input), but it does not provide any certificates.

We computed $\Delta(S)$ with coefficients in $\QQ$ for $S$ running over various complete bipartite graphs $K_{m,n}$, 
where the nodes of the two sides are labeled $\{1,\dotsc,m\}$ and $\{m+1,\dotsc,n\}$.
Note that $K_{m,n}$ is not close to being shifted; in fact, it is being equally far from being shifted w.r.t.\ the lexicographic and the reverse lexicographic order.
While $K_{m,n}$ and $K_{n,m}$ are clearly isomorphic, their node labelings are different, and so the computation times may differ, too.
We obtain by $\Delta(S)$ by computing one of
\begin{enumerate}
	\item $\Delta(S) = \Delta_{\frX}(S)$ for the matrix $\frX = (x_{ij})_{ij}$ over the field $\EE = \QQ(x_{ij} \mid 1 \leq i,j \leq n)$, where $n$ is the number of vertices of $S$,
	\item $\Delta(S) = \Delta_{\frR(w_0)}(S)$ for the matrix $\frR(w_0)$ over $\EE = \QQ(x_{ij}\mid1<i<j\leq n$),
	\item using the Las Vegas algorithm~\ref{algo:lv}, which uses $\frR(w_0)$ for checking if the random $\QQ$-matrix $u$ computes $\Delta(S)$ correctly.%
\end{enumerate}

The time and memory consumption of the three approaches to computing $\Delta(S)$ for different graphs $S$ is listed in \cref{tab:timing-graphs}.
Note that each algorithm involves computing the row echelon form of a matrix whose entries are multivariate polynomials over $\QQ$.
We list the maximal length (i.e., number of terms) and degree of the entries of this matrix ($\frX$, $\frR(w_0)$, or the submatrix $m$ of $\frR(w_0)$ as defined in \cref{algo:verify}, respectively) after computing the row echelon form
(initially, both quantities are 2 for all algorithms).
For the Las Vegas algorithm, it is not always necessary to compute this row echelon form (namely, if $T=\emptyset$ in \cref{algo:verify});
in this case, we print \enquote{n/a}.
The table also lists (in gray) running times of these algorithms using the lazy \cref{alg:lazy reduction} for computing the row echelon form; see \cref{subsubsec:lazy} for details.

The results show the superiority of the Las Vegas method, with memory as the limiting factor.
Even larger examples (like, e.g., $K_{4,6}$ and $K_{6,4}$ with ten nodes and 24 edges each) only require a few seconds; the largest one in our list, $K_{6,5}$ with eleven nodes and 30 edges takes less than two minutes.

\begin{table*}
	\caption{(bipartite graphs) %
		Comparison of the runtime (in seconds) and accumulated memory consumption (in \unit{\mega\byte}) of computing $\Delta(S)$ (with coefficients in $\QQ$) for various complete bipartite graphs $S$
		using the deterministic algorithm (using the matrices $\frX$ and $\frR(w_0)$), and the Las Vegas algorithm (\cref{algo:lv});
    see \cref{sec:bipartite}.
		The columns printed in gray are obtained by using \cref{alg:lazy reduction} for finding $\Ind g$, instead of computing the full row echelon form;
		see \cref{subsubsec:lazy}.
	}
	\label{tab:timing-graphs}
	\centering
	\setlength{\tabcolsep}{2pt}
	\tiny
	\leavevmode\clap{
	\csvreader[
		tabular = {
			c A[2.0] A[2.0]
			A[3.2] >{\color{gray}}A[3.2] M[3.2] >{\color{gray}}M[3.2] A[5.0] A[2.0]
			A[3.2] >{\color{gray}}A[2.2] M[3.2] >{\color{gray}}M[3.2] A[5.0] A[2.0]
			A[3.2] >{\color{gray}}A[3.2] M[3.2] >{\color{gray}}M[3.2] A[5.0] A[2.0]
		},
		head to column names,
		table head = {
			\toprule
			& & & \multicolumn{6}{c}{$\Delta_{\frX(S)}$} & \multicolumn{6}{c}{$\Delta_{\frR(w_0)}(S)$} & \multicolumn{6}{c}{Las Vegas} \\
			\cmidrule(lr){4-9} \cmidrule(lr){10-15} \cmidrule(lr){16-21}
			$S$ & \Rot{\#vertices} & \Rot{\#edges} &
			\Rot{time} & \Rot{time'} & \Rot{memory} & \Rot{memory'} & \Rot{length} & \Rot{degree} &
			\Rot{time} & \Rot{time'} & \Rot{memory} & \Rot{memory'} & \Rot{length} & \Rot{degree} &
			\Rot{time} & \Rot{time'} & \Rot{memory} & \Rot{memory'} & \Rot{length} & \Rot{degree} \\
			\midrule
		},
		table foot = \bottomrule
	]{bipartite_table_born_post_issac.csv}{}{%
		\instance & \nVertices & \nEdges &
		\avTime & \avfTime & \avMemory & \avfMemory & \avmaxLenAfter & \avmaxDegAfter &
		\hvTime & \hvfTime & \hvMemory & \hvfMemory & \hvmaxLenAfter & \hvmaxDegAfter &
		\lvTime & \lvfTime & \lvMemory & \lvfMemory & \lvmaxLenAfter & \lvmaxDegAfter
	}
	}
\end{table*}

\subsection{Surface triangulations}\label{subsec:surfaces}
\newdimen\HalfHourOrOotDimen
\newcommand{\HalfHourOrOot}[2]{
	\IfDecimal{#1}{%
		\ifthenelse{\lengthtest{#1pt>1800pt}}{oot}{#2}
	}{%
		#1
	}
}
\begin{table*}
	\caption{(surfaces) %
		Computing the $2$-dimensional part of $\Delta(K)$ using the Las Vegas algorithm,
		where $K$ ranges over triangulations of surfaces of varying genus $g$, orientability $o$, and number $n$ of vertices \cite{Lutz:surfaces}.
		Triangulations of identical parameters are numbered consecutively (column \enquote{$i$}).
		Runs marked with an asterisk did not run the loop in \cref{algo:verify}, line~\ref{algo:verify:loop}, because $T=\emptyset$.
		The gray numbers were obtained using \cref{alg:lazy reduction} for computing row echelon forms.
		This experiment uses \OSCAR 1.7.0.
	}
	\label{tab:timing-surfaces}
	\centering
	\setlength\tabcolsep{1.5pt}
	\tiny%
	\renewcommand{\arraystretch}{0.965}
	\leavevmode\clap{
		\csvreader[
		head to column names,
		head to column names prefix=Col,
		tabular = {
			r |
			*{5}{r}
			| c A[ 1] c  
			| c A[ 3] c  
			| ccA[ 3] c  
			| c A[ 3] c  
			| c A[ 1] c  
			| c A[ 2] c  
			| c A[ 1] c  
			| c A[ 1] c  
			| c A[ 1] c  
		},
		table head = {
			\toprule
			&&&&&&
			\multicolumn{3}{c}{$\QQ$} &
			\multicolumn{3}{c}{$\GF{2}$} &
			\multicolumn{4}{c}{$\GF{4}$} &
			\multicolumn{3}{c}{$\GF{3}$} &
			\multicolumn{3}{c}{$\GF{9}$} &
			\multicolumn{3}{c}{$\GF{5}$} &
			\multicolumn{3}{c}{$\GF{25}$} &
			\multicolumn{3}{c}{$\GF{7919}$} &
			\multicolumn{3}{c}{$\GF{7919^2}$} \\
			& $n$ & $o$ & $g$ & $i$ & \Rot{\#$2$-faces} &
			\Rot{time} & \Rot{\#trials} &&
			\Rot{time} & \Rot{\#trials} &&
			\Rot{time} && \Rot{\#trials} &&
			\Rot{time} & \Rot{\#trials} &&
			\Rot{time} & \Rot{\#trials} &&
			\Rot{time} & \Rot{\#trials} &&
			\Rot{time} & \Rot{\#trials} &&
			\Rot{time} & \Rot{\#trials} &&
			\Rot{time} & \Rot{\#trials} &  \\
			\midrule
		},
		table foot = \bottomrule,
		respect all, 
		filter equal={\Colq}{2}
		]{surfaces_combined.csv}{}{%
			\thecsvrow & \ColnVertices & \Tick{\Colorientable} &\Colgenus & \Colindex & \ColnFaces &
			\TimeOrDash[3.1]{ColQQLv}        & \CS{ColQQLvTrials}        & \EarlyReturn{ColQQ}        &
			\TimeOrDash[4.1]{ColF2Lv}        & \CS{ColF2LvTrials}        & \EarlyReturn{F2}        &
			\TimeOrDash[4.1]{ColF4Lv}        & \textcolor{gray}{\HalfHourOrOot{\CS{Col_F4LvfTime}}{\TimeOrDash[3.1]{Col_F4Lvf}}} &	\CS{ColF4LvTrials}        & \EarlyReturn{ColF4} \textcolor{gray}{\EarlyReturn{Col_F4}}       &
			\TimeOrDash[4.1]{ColF3Lv}        & \CS{ColF3LvTrials}        & \EarlyReturn{CF3}        &
			\TimeOrDash[4.1]{ColF9Lv}        & \CS{ColF9LvTrials}        & \EarlyReturn{CF9}        &
			\TimeOrDash[3.1]{ColF5Lv}        & \CS{ColF5LvTrials}        & \EarlyReturn{CF5}        &
			\TimeOrDash[3.1]{ColF25Lv}       & \CS{ColF25LvTrials}       & \EarlyReturn{CF25}       &
			\TimeOrDash[4.1]{ColF7919Lv}     & \CS{ColF7919LvTrials}     & \EarlyReturn{CF7919}     &
			\TimeOrDash[4.1]{ColF62710561Lv} & \CS{ColF62710561LvTrials} & \EarlyReturn{CF62710561}
		}
	}
\end{table*}
\begin{table*}
	\centering
	\caption{(surfaces, detail) %
		Detail for selected rows of \cref{tab:timing-surfaces}:
		The total running time (\enquote{tot}) of the Las Vegas algorithm as reported in the column \enquote{time} of \cref{tab:timing-surfaces}
		essentially, time comprises the time for computing $\Delta_g(S)$ for the $N = 1$ (for $\QQ$) resp.\ $N = 500$ (for finite fields) random matrices $g$ (column \enquote{A}),
		and the time for verifying the correctness of the result using \cref{algo:verify} (column \enquote{B}).
		The row numbers refer to the same instances as \cref{tab:timing-surfaces}.
	}
	\label{tab:timing-surfaces:phases}
	\setlength\tabcolsep{1pt}
	\tiny
	\leavevmode\clap{
		\csvreader[
		head to column names,
		head to column names prefix=Col,
		tabular = {
			r
			| A[3.1] A[1.1] A[4.1]  
			| A[4.1] A[1.1] A[4.1]  
			| A[4.1] A[1.1] A[4.1]  
			| A[4.1] A[1.1] A[4.1]  
			| A[4.1] A[1.1] A[4.1]  
			| A[3.1] A[1.1] A[3.1]  
			| A[3.1] A[1.1] A[3.1]  
			| A[3.1] A[1.1] A[3.1]  
			| A[3.1] A[1.1] A[3.1]  
		},
		table head = {
			\toprule &
			\multicolumn{3}{c|}{$\QQ$} &
			\multicolumn{3}{c|}{$\GF{2}$}&
			\multicolumn{3}{c|}{$\GF{4}$}&
			\multicolumn{3}{c}{$\GF{3}$}&
			\multicolumn{3}{c}{$\GF{9}$}&
			\multicolumn{3}{c}{$\GF{5}$}&
			\multicolumn{3}{c}{$\GF{25}$}&
			\multicolumn{3}{c}{$\GF{7919}$}&
			\multicolumn{3}{c}{$\GF{7919^2}$}\\
			&{tot} & {A} & {B}
			&{tot} & {A} & {B}
			&{tot} & {A} & {B}
			&{tot} & {A} & {B}
			&{tot} & {A} & {B}
			&{tot} & {A} & {B}
			&{tot} & {A} & {B}
			&{tot} & {A} & {B}
			&{tot} & {A} & {B} \\
			\midrule
		},
		table foot = \bottomrule,
		filter expr={ 
			(
			test{\ifnumequal{\Index}{12}}
			or test{\ifnumequal{\Index}{49}}
			or test{\ifnumequal{\Index}{51}}
			or test{\ifnumequal{\Index}{54}}
			)	    
			and test{\ifnumequal{\Colq}{2}}
		},
		]{surfaces_table_post_issac_born.csv}{1=\Index}{%
			\Index &
			\CS{ColQQLvTime}        & \CS{ColQQLvTimeA}        & \CS{ColQQLvTimeB}        &
			\CS{ColF2LvTime}        & \CS{ColF2LvTimeA}        & \CS{ColF2LvTimeB}        &
			\CS{ColF4LvTime}        & \CS{ColF4LvTimeA}        & \CS{ColF4LvTimeB}        &
			\CS{ColF3LvTime}        & \CS{ColF3LvTimeA}        & \CS{ColF3LvTimeB}        &
			\CS{ColF9LvTime}        & \CS{ColF9LvTimeA}        & \CS{ColF9LvTimeB}        &
			\CS{ColF5LvTime}        & \CS{ColF5LvTimeA}        & \CS{ColF5LvTimeB}        &
			\CS{ColF25LvTime}       & \CS{ColF25LvTimeA}       & \CS{ColF25LvTimeB}       &
			\CS{ColF7919LvTime}     & \CS{ColF7919LvTimeA}     & \CS{ColF7919LvTimeB}     &
			\CS{ColF62710561LvTime} & \CS{ColF62710561LvTimeA} & \CS{ColF62710561LvTimeB}
		}
	}
\end{table*}
In our second experiment we look into $2$-dimensional simplicial complexes with nontrivial topology.
We computed $\Delta(K)$ using the Las Vegas algorithm~\ref{algo:lv}, where $K$ runs over various triangulations of surfaces, with coefficients in $\QQ$ and various finite fields.
The triangulations were obtained from Frank Lutz' list of manifold triangulations \cite{Lutz:surfaces}\footnote{%
  The files in \cite{Lutz:surfaces} contain the facets of the respective triangulations.
  The labels are of the form \texttt{manifold\_lex\_d$d$\_n$n$\_o$o$\_g$g$\_\#$i$},
  where $d$ stands for the dimension (2), $n$ for the number of vertices, $o$ for the orientability (1 or 0), $g$ for the genus,
  and $i$ is the same consecutive numbering as in \cref{tab:timing-surfaces}.
  We include this data as \texttt{mrdi} files in the \texttt{examples} folder of \cite{BenchmarksRepo},
  where we change the filenames by removing the \# and adding the corresponding table entry number to the front.
}.

We report the running times for computing $\Delta(K^{(2)})$ in \cref{tab:timing-surfaces}.
The triangulations are classified by genus, orientability, number of vertices, and (in case of ambiguity) a consecutive number.
As explained in \cref{sec:lv-implementation-detail}, the Las Vegas algorithm picks random matrices $g_1,\dotsc,g_N$ over $\FF$, where $N = 1$ for $\FF = \QQ$ and $N = 500$ otherwise.
The algorithm then tests if $g \coloneqq \lex\argmin\Set{\Delta_{g_i}(K^{(2)}); g_i, i = 1,\dotsc,N}$ satisfies $\Delta(K^{(2)}) = \Delta_g(K^{(2)})$.
\Cref{tab:timing-surfaces} contains in the column \enquote{\#trials} the index $i$ of the first $g_i$ that attains this minimum. 
We mark a run in \cref{tab:timing-surfaces} with an asterisk if $T = \emptyset$ in \cref{algo:verify}.
In this case, verifying if the shift $\Delta_g(K^{(2)})$ did not involve any computation over the fraction field $\EE = \FF(x_{\inv w_0})$.

For selected instances from \cref{tab:timing-surfaces}, we report in \cref{tab:timing-surfaces:phases} on the time needed to compute $g \coloneqq \lex\argmin\Set{\Delta_{g_i}(K^{(2)}); g_i, i = 1,\dotsc,N}$ (\enquote{phase A})
and verify that $\Delta(K^{(2)}) = \Delta_g(K^{(2)})$ (\enquote{phase B}).
For the runs marked with an asterisk in \cref{tab:timing-surfaces}, there is no phase B.

\medskip
The results of \cref{tab:timing-surfaces,tab:timing-surfaces:phases} lead us to a few observations.
Most notably, we observe that if $\FF$ is a small finite field, the run time (and also memory consumption) can be much larger than for $\FF = \QQ$ or a large finite field.
We observed that the run time and memory consumption are tightly correlated.
In the interest of a tight representation, we omit precise numbers for memory consumption.

All instances of reasonable size that can be computed within seconds are fast because in fact,
\cref{algo:verify} essentially does nothing:
in these cases (marked with an asterisk in \cref{tab:timing-surfaces}), the set $T$ in \cref{algo:verify} is empty,
which means that the algorithm need not do any expensive calculations over the fraction field $\EE$.
In cases where the verification process actually needs to do such calculations (instances without asterisk),
we see in \cref{tab:timing-surfaces:phases} that the total run time is dominated by the verification step.

As far as the coefficients are concerned, we picked the first few finite fields of prime order, namely $\GF{2}$, $\GF{3}$ and $\GF{5}$ plus $\GF{7919}$ as one fairly large finite field; the particular prime 7919 does not seem to be important.
In \cref{ex:not every generic shift is concrete} we saw that for some input a particular field may be too small to allow for a sufficiently generic matrix of the correct size.
Accordingly, we will see some cases in \cref{sec:non-surfaces} below where 500 trials were not enough to find a sufficiently generic matrix.
Therefore, we experimented with extending the ground fields; these are the columns with $\GF{4}$, $\GF{9}$, $\GF{25}$ and $\GF{7919^2}$.
We observe that if $\FF$ is one of $\QQ$, $\GF{7919}$ or $\GF{7919^2}$, every instance that could be computed was actually obtained as $\Delta_g(K^{(2)})$ for the first pick for $g$.
For smaller fields, this is not the case.
However, passing from the prime field $\GF{p}$ to the algebraic extension $\GF{p^2}$ makes it much easier to find a generic matrix;
we see this by the numbers in the columns \enquote{\#trials} decreasing when passing from $\GF{p}$ to $\GF{p^2}$.
In some cases in \cref{sec:non-surfaces} below, where $p\in\{2,3\}$, this trick allows us to compute shifts, which are inaccessible otherwise.
Of course, computing in a more complicated finite field can come, depending on the instance, at the expense of longer computations and greater memory consumption in phase A.

Concerning topology, we note that we observe much greater run times and memory consumption for $\Char\FF = 2$ compared to the other fields
for the non-orientable surface triangulations, where as for orientable surfaces (i.e., the ones without 2-torsion in homology),
no such special role of characteristic two is observable.
We also note that $T = \emptyset$ for $\FF \in \{\GF{2}, \GF{4}\}$ in \cref{algo:verify} if and only if the surface has genus zero.

\begin{remark}
	\label{rmk:increased-time}
	Comparing with Table 2 in the ISSAC extended abstract displays the benefit of our improvements in \cref{algo:verify}.
	Now the instances of genus zero can be computed in a few seconds at most.
	However, while almost all instances can be computed faster by varying margins, we note that the instance in row 25 of \cref{tab:timing-surfaces} stands out
	by an running time \emph{increased} by a factor of 1.84.
	We explain this by the intricacies of computing row echelon forms over fraction fields:
	the algorithm for computing row echelon forms we use (see \cref{rmk:ref-implementation}) computes the greatest common divisor of each row of the matrix after each operation.
	The computation of the gcd of \emph{more} numbers may be \emph{faster}, because, for example, the computation can exit early if two numbers are coprime.
	Indeed, for the instance in row 25 of \cref{tab:timing-surfaces}, the algorithm we used here needed \emph{longer} to compute the row echelon form of the matrix $m$ from \cref{algo:verify}
	than for the corresponding matrix underlying Table 2 of the ISSAC extended abstract, of which $m$ is a submatrix.
\end{remark}
\medskip
For the sake of completeness, we have also made a comparison with our implementation of the Monte Carlo algorithm with the \Macaulay based implementation of Keehn \cite{M2ExtShifting}.
We ran our comparison on the examples of \cref{tab:timing-surfaces}.
The statistics (measured in seconds) of the running times of our implementation are the following,
average 0.000661, longest  0.00108, shortest 0.000138, median 0.000671.
The implementation of Keehn has the following statistics, average 0.09763, longest 0.12931, shortest 0.02138, median 0.10941.
Notably the implementation in \OSCAR is two orders of magnitude faster.

\subsection{Moore spaces}
\label{sec:non-surfaces}
Our computations for surfaces revealed a correlation between the running times and the homology of a surface.
To study this behavior further, we computed $\Delta(K^{(2)})$ using the Las Vegas algorithm~\ref{algo:lv}, where $K$ runs over triangulations of Moore spaces $M(G, 1)$, for various finite cyclic groups $G$.
Recall that a Moore space $M(G,n)$ is a cell complex with $\widetilde{H}_n(M(G,n)) = G$ and $\widetilde{H}_k(M(G,n)) = 0$ for $k \neq n$, where $\widetilde{H}_n(X)$ denotes the $n$th reduced homology of a topological space $X$; see \cite{Hatcher2002} for more on Moore spaces.
In our experiments, we considered triangulations of $M(\ZZ/q\ZZ, 1)$ for $q = 2,3,4,5$.
In a way, these are the easiest topological spaces with nontrivial homology.
Concretely, cell complexes for these spaces are obtained by identifying the edges of a $q$-gon.
The triangulations we considered are then obtained by the double barycentric subdivision of this cell complex,
and then applying \texttt{polymake}'s function \texttt{bistellar\_simplicification} \cite{DMV:polymake}, which uses simulated annealing to reduce the number of vertices as proposed by Björner and Lutz \cite{Bjoerner+Lutz:2000}.
We do not attempt to verify if these triangulations are minimal.
The precise triangulations are available as \texttt{mrdi} files \cite{DellaVecchiaJoswigLorenz:2024} in our repository \cite{BenchmarksRepo}.\footnote{%
  The files can be found in \texttt{examples/non\_surfaces} and can be loaded with \OSCAR.
}

\begin{table*}
	\caption{(Moore spaces)
		Run time data for the Moore space data set (see \cref{sec:non-surfaces}.
		The time limit was three hours per instance.
		For the meaning of the other columns (including the asterisks), see \cref{tab:timing-surfaces}.
	}
	\label{tab:timing-non-surfaces}
	\setlength\tabcolsep{2pt}
	\centering
	\tiny
	\leavevmode\clap{
	\csvreader[
	head to column names,
	head to column names prefix=Col,
	tabular = {
		*{4}{c}
		| c A[ 1] c 
		| c A[<3] c 
		| c A[ 2] c 
		| c A[ 3] c 
		| c A[ 2] c 
		| c A[ 2] c 
		| c A[ 1] c 
	},
	table head = {
		\toprule
		&&&&
		\multicolumn{3}{c}{$\QQ$} &
		\multicolumn{3}{c}{$\GF{2}$} &
		\multicolumn{3}{c}{$\GF{4}$} &
		\multicolumn{3}{c}{$\GF{3}$} &
		\multicolumn{3}{c}{$\GF{9}$} &
		\multicolumn{3}{c}{$\GF{5}$} &
		\multicolumn{3}{c}{$\GF{25}$} \\
		instance & $H_1$ & \Rot{\rlap{\#vertices}} & \Rot{\rlap{\#$2$-faces}} &
		\Rot{time} & \Rot{\#trials} &  &
		\Rot{time} & \Rot{\#trials} &  &
		\Rot{time} & \Rot{\#trials} &  &
		\Rot{time} & \Rot{\#trials} &  &
		\Rot{time} & \Rot{\#trials} &  &
		\Rot{time} & \Rot{\#trials} &  &
		\Rot{time} & \Rot{\#trials} &  \\
		\midrule
	},
	table foot = \bottomrule,
	respect none, 
	respect sharp,
	respect underscore,
	filter equal={\Colq}{2}
	]{non_surfaces_table_born_post_issac.csv}{}{%
		\texttt{\Colinstance}   & $\FormatGroup{\CS{ColH1}}$ & \ColnVertices & \ColnFaces &
		\TimeOrDash[1.2]{ColQQLv}      & \CS{ColQQLvTrials}        & \EarlyReturn{ColQQ}        &
		\TimeOrDash[4.0]{ColF2Lv}      & \CS{ColF2LvTrials}        & \EarlyReturn{ColF2}        &
		\TimeOrDash[4.0]{ColF4Lv}      & \CS{ColF4LvTrials}        & \EarlyReturn{ColF4}        &
		\TimeOrDash[3.0]{ColF3Lv}      & \CS{ColF3LvTrials}        & \EarlyReturn{ColF3}        &
		\TimeOrDash[3.0]{ColF9Lv}      & \CS{ColF9LvTrials}        & \EarlyReturn{ColF9}        &
		\TimeOrDash[4.0]{ColF5Lv}      & \CS{ColF5LvTrials}        & \EarlyReturn{ColF5}        &
		\TimeOrDash[4.0]{ColF25Lv}     & \CS{ColF25LvTrials}       & \EarlyReturn{ColF25}       
	}
	}
\end{table*}
These examples were run with the same memory limit as above, and a time limit of three hours.
Except for $\EE = \QQ$, 500 samples were taken for the random matrix $g$.
The results, which can be found in \cref{tab:timing-non-surfaces}, are arranged in the same way as the results for the surface triangulations.
In particular for characteristic two, we observe that no instance except two could be computed (with 500 samples) at all over $\GF{2}$,
while when passing to $\GF{4}$, these become computable.
Also, we observe that the index of the first random matrix that yielded the correct shifted complex
decreases as we pass to higher prime powers.
Note that there is no guarantee that a particular instance can be computed over a particular prime field with the Las Vegas algorithm \emph{at all};
in fact, \cref{ex:not every generic shift is concrete} shows that passing to a bigger field may be necessary.

For the instances reported on here, we see that it may happen that $N=500$ are not enough samples to pick a random matrix $g$ such that $\Delta(K^{(q)}) = \Delta_g(K^{(q)})$;
in this case, the columns \enquote{\#trials} reads \enquote{$>500$} and,
since we have not found $\Delta(K^{(q)})$ in this case, we do not report a running time for this case.
Nevertheless, we observe it is usually quick to discard all sampled matrices $g$ (because \cref{algo:verify} has $T=\emptyset$; see the instances marked with an asterisk).

\subsection{Selected triangulations of other spaces}
\label{sec:balls-and-spheres}

In this section, we report the running times for various triangulations of balls, spheres and a few other spaces.
These measurements do not appear in the ISSAC extended abstract because our implementation then was unable to complete the computation within the given limits. 
Now, thanks to our improved \cref{algo:verify}, most instances take only seconds.

Except for the first, which is a non-polytopal $3$-sphere from \cite{altshuler}, our instances were obtained from \cite{Lutz:surfaces}.
These are selected simplicial complexes of dimension at most four with 14 vertices or fewer.
More precisely, \verb+altshuler_M_9_963+ and \verb+barnette_sphere+ are $3$-dimensional spheres; the \verb+dunce_hat+ is a contractible but not collapsible $2$-complex; \verb+dunce_hat_in_3_ball+ is a $3$-ball with a dunce hat as an embedded subcomplex; the space \verb+kuehnel_cp2_9+ is a complex projective plane, which is an orientable $4$-manifold; \verb+l_3_1_brehm+ is a lens space; \verb+rp_3_11+ is a real projective $3$-space; and \verb+rudin+ is a non-shellable $3$-ball.
The precise triangulations are available as \texttt{mrdi} files \cite{DellaVecchiaJoswigLorenz:2024} in our repository \cite{BenchmarksRepo}.\footnote{%
  The files can be found in \texttt{examples/other\_examples} and can be loaded with \OSCAR.
}
For a complex $K$ listed in \cref{tab:balls-and-spheres}, we report the running times for computing $\Delta(K^{(q)})$ for $q = 1,\dotsc,\dim K$.
Otherwise, the table is structured as \cref{tab:timing-surfaces}.
\newif\ifNextblock
\begin{table*}
	\caption{(other spaces)
		Time to compute the shift of various spaces for various fields, using the Las Vegas algorithm.
		The meaning of the columns \enquote{\#trials} and of the asterisks is as in \cref{tab:timing-surfaces}.
		The time limit was three hours per instance.
	}
	\label{tab:balls-and-spheres}
	\setlength\tabcolsep{2pt}
	\centering
	\tiny
	\leavevmode\clap{
	\def\X{}
	\csvreader[
		head to column names,
		head to column names prefix=Col,
		tabular = {
			*{5}{c}
			| c A[ 1] c 
			| c A[<3] c 
			| c A[<3] c 
			| c A[ 3] c 
			| c A[ 2] c 
			| c A[<3] c 
			| c A[ 1] c 
			| c A[ 1] c 
			| c A[ 1] c 
		},
		table head = {
			\toprule
			&&&&&
			\multicolumn{3}{c}{$\QQ$} &
			\multicolumn{3}{c}{$\GF{2}$} &
			\multicolumn{3}{c}{$\GF{4}$} &
			\multicolumn{3}{c}{$\GF{3}$} &
			\multicolumn{3}{c}{$\GF{9}$} &
			\multicolumn{3}{c}{$\GF{5}$} &
			\multicolumn{3}{c}{$\GF{25}$} &
			\multicolumn{3}{c}{$\GF{7919}$} &
			\multicolumn{3}{c}{$\GF{7919^2}$} \\[2pt]
			instance &\Rot{\#vertices} & $q$ & \Rot{{\#$q$-faces}} & $H_{q-1}$ &
			\Rot{time} & \Rot{\#trials} &  &
			\Rot{time} & \Rot{\#trials} &  &
			\Rot{time} & \Rot{\#trials} &  &
			\Rot{time} & \Rot{\#trials} &  &
			\Rot{time} & \Rot{\#trials} &  &
			\Rot{time} & \Rot{\#trials} &  &
			\Rot{time} & \Rot{\#trials} &  &
			\Rot{time} & \Rot{\#trials} &  &
			\Rot{time} & \Rot{\#trials} &  \\
			\midrule
		},
		table foot = \bottomrule,
		respect none, 
		respect sharp,
		respect underscore,
		after line={\ifthenelse{\equal{\Colq}{\Coldim}}{\global\Nextblocktrue}{\global\Nextblockfalse}},
		late after line={\ifNextblock\\[3pt]\else\\\fi},
		late after last line={\\}
	]{other_examples_table_dragon-1.csv}{}{%
		\ifthenelse{\equal{\Colq}{1}}{\multirow{\Coldim}{*}{\texttt{\StrSubstitute{\Colinstance}{.mrdi}{}}}}{} &
		\ifthenelse{\equal{\Colq}{1}}{\multirow{\Coldim}{*}{\ColnVertices}}{} &
		\Colq & 
		\ColnFaces &
		$\FormatGroup{\CS{ColHq-1}}$ 
		& 
		\TimeOrDash[1.2]{ColQQLv}      & \CS{ColQQLvTrials}        & \EarlyReturn{ColQQ}        &
		\TimeOrDash[1.2]{ColF2Lv}      & \CS{ColF2LvTrials}        & \EarlyReturn{ColF2}        &
		\TimeOrDash[2.1]{ColF4Lv}      & \CS{ColF4LvTrials}        & \EarlyReturn{ColF4}        &
		\TimeOrDash[1.2]{ColF3Lv}      & \CS{ColF3LvTrials}        & \EarlyReturn{ColF3}        &
		\TimeOrDash[3.0]{ColF9Lv}      & \CS{ColF9LvTrials}        & \EarlyReturn{ColF9}        &
		\TimeOrDash[1.2]{ColF5Lv}      & \CS{ColF5LvTrials}        & \EarlyReturn{ColF5}        &
		\TimeOrDash[3.0]{ColF25Lv}     & \CS{ColF25LvTrials}       & \EarlyReturn{ColF25}       &
		\TimeOrDash[1.2]{ColF7919Lv}   & \CS{ColF7919LvTrials}     & \EarlyReturn{ColF7919}     &
		\TimeOrDash[2.1]{ColF62710561} & \CS{ColF62710561LvTrials} & \EarlyReturn{ColF62710561}
	}
}
\end{table*}

\subsection{Lazy row reduction}\label{subsubsec:lazy}
As mentioned in \cref{sec:fabian's reduction}, we also consider an alternative, lazy algorithm (\cref{alg:lazy reduction}) for computing row echelon forms over fraction fields.
This has an impact on the running time and memory footprint both of the deterministic and the Las Vegas algorithm.
\Cref{tab:timing-graphs} shows in gray the running times and memory footprint of this approach.
Additionally, we report in \cref{tab:timing-surfaces} (also in gray) the running times of this algorithm for selected surface triangulations for the coefficient field $\GF{4}$,
The respective instances in \cref{tab:timing-surfaces} are selected so that the \cref{algo:verify} has to actually execute line~\ref{algo:verify:exit},
i.e., the instance does not have an asterisk in the table.
	
From \cref{tab:timing-graphs}, we see that while computing $\Delta(S)$ deterministically (using $\frX$ or $\frR(w_0)$) noticeably benefits from this algorithm,
the Las Vegas algorithm does not profit consistently from using the lazy reduction.

From \cref{tab:timing-surfaces}, we see that there can be huge differences (in both directions) between the run times of the standard and the lazy variant of the Las Vegas algorithm.
Devising efficient algorithms for computing row echelon forms of matrices over fields of rational functions is a very active area of research on its own,
and developing heuristics to decide which algorithm is most suitable for which instance in the context of algebraic shifting is beyond the scope of this paper.

\section{Concluding remarks} 
Our results can be summarized as follows.
From a computational point of view, algebraic shifting decomposes into three different regimes.
First, for shifting general hypergraphs and simplicial complexes in characteristic zero, the best option is to use the traditional Monte Carlo algorithm, with few samples.
Secondly, in positive characteristic, the best strategy is to use the Las Vegas algorithm sampling from a finite field extension when the field is small and using either eager or lazy reduction.
Thirdly, combinatorial algorithms for special scenarios like, e.g., the algorithms of Keehn and Nevo \cite{KeehnNevo:2024} for surfaces of low genus and fields of characteristic zero.

We sketch directions for further research.
It is clear that the implementations of our algorithms would benefit from improvements in row reduction algorithms for matrices of multivariate polynomials.
While optimized algorithms of row reducing matrices of univariate polynomials exist \cite{StorjohannVillard05, JeannerodVillard06},
the linear algebra with multivariate polynomials as coefficients requires more research.
Is there a good algorithm for row reduction that exploits the structure of compound matrices?

Apart from improvements in the direction of row reduction, how can the topology or any other properties of the simplicial complex be exploited?
The work of Keehn and Nevo \cite{KeehnNevo:2024} have found ways of excluding certain simplices in the shifted complex depending on the topology.
The computations in \cref{sec:balls-and-spheres} were fast for examples with homology concentrated in the top dimension.
How can this be generalized?

Lastly, when computing the full shift, we use the longest element $w_0$.
We have noticed that it is sometimes sufficient to use a smaller word for a given simplicial complex.
Using a smaller word $w$ would imply less memory consumption since $\frR(w)$ contains fewer indeterminants than $\frR(w_0)$; cf.\ \cref{thm:left-invariance}.
Is there a way to pick such a $w$? Or rather a sequence of $w_i$ such that $\Delta_{\frR(w_t)}( \cdots (\Delta_{\frR(w_1)}(S))\cdots) = \Delta(S)$?

{\sloppy\printbibliography}


\end{document}